\documentclass[reqno]{amsart}
\usepackage{amssymb}
\usepackage{mathrsfs}
\usepackage{amsmath,amstext,amsfonts,verbatim}

\usepackage[latin1]{inputenc}

\usepackage{amsmath,amsthm,amssymb}

\usepackage{amsfonts}

\usepackage{amsmath,amssymb}
\usepackage{graphicx}
\usepackage{color}
\usepackage{indentfirst}
\usepackage{cite}

\newtheorem{theorem}{Theorem}[section]
\newtheorem{lemma}[theorem]{Lemma}

\theoremstyle{definition}
\newtheorem{definition}[theorem]{Definition}

\newtheorem{prop}{Proposition}[section]
\newtheorem{Corollary}[theorem]{Corollary}
\newtheorem{fact}[theorem]{Fact}
\newtheorem{question}[theorem]{Question}

\theoremstyle{remark}
\newtheorem{remark}[theorem]{Remark}

\numberwithin{equation}{section}

\newcommand{\neweq}[1]{\begin{equation}\label{#1}}

\def\phi{\varphi}

\def\incep{\left\{\begin{array}{cl} }
 \def\termin{\end{array}\right. }
\def\2af{2^*_\alpha}

\title [Weyl mean equicontinuity and Weyl mean sensitivity of a random dynamical system]
{\textbf{Weyl mean equicontinuity and Weyl mean sensitivity of a random dynamical system}}

\author{Yuan Lian}
\address{College of Mathematics and Statistics, Taiyuan Normal University, Taiyuan 030619, China}
\email{andrea@tynu.edu.cn}


\subjclass[]{}

\date{}

\begin{document}

\begin{abstract}
In this article, we introduce the concepts of Weyl-mean equicontinuity and Weyl-mean sensitivity of a random dynamical system associated to an infinite countable discrete
amenable group action. We obtain the dichotomy result to Weyl-mean equicontinuity and Weyl-mean sensitivity of a random dynamical system when the corresponding skew product transformation is minimal and $\Omega$ is finite.
\end{abstract}

\maketitle

\section{introduction}
The theoretical basis for amenable group actions should be returned to the pioneering paper \cite{OW} by Ornstein and Weiss, and it was widely developed by Rudolph and Weiss
\cite{RW} and Danilenko \cite{D}. See also Benjy Weiss's famous article \cite{W}. The theory of random transformation is more suitable for studying the systems which are
formed by the iterations of different mappings than the classical dynamical systems which are only formed by the iterations of one mapping. The construction of this
theoretical structure comes from the article of Ulam and von Neumann \cite{Uv}. A few years later, the random ergodic theorem was established and proved by Kakutani \cite{K1}.
In the 1970s, their work continued within the framework of the relative ergodic theory (see citation \cite{LW,T}), but all this attracted only a little interest. The
emergence of stochastic flows as solutions to stochastic differential equation has given the subject a huge impetus.

When the other side of the concept of sensitivity is taken into account, the concept of equicontinuity at a point naturally arises, see \cite{GW}. It is well
known that equicontinuous systems have simple dynamical behaviors. If a mapping collection defined by the action of a group is a family of uniformly equicontinuous, then the
dynamical system is called equicontinuous systems. Equicontinuous system is the simplest dynamic system; in fact,
equicontinuous minimal systems have a complete classification.

Because mean equicontinuity is related to the ergodicity of measurable dynamical systems(i.e. dynamical systems with invariant probability measures),
it has attracted extensive interest of scholars in recent years. In particular, it has been shown that using a version of the measure theory of mean equicontinuity,
it is possible to characterize some special cases when a measure-preserving system has a discrete spectrum \cite{G1} and when the maximal equicontinuous factor is
actually an isomorphism (see \cite{BLM, LTY1})

There are two variants of the concept of mean equicontinuity: one is called Weyl mean equicontinuity, and the other is called Besicovitch mean equicontinuity.
The concepts of Weyl-and Besicovitch-mean equicontinuity are introduced, which are used for in \cite{LTY1} for $\mathbb{Z}$ actions and in \cite{FGL} for amenable group
actions, respectively. Inspired by the idea, we will introduce the concepts of Weyl-mean equicontinuity and Weyl-mean sensitivity of a random dynamical system associated
to an infinite countable discrete amenable group action.

We know that if a dynamical system $(X,T)$ is almost equicontinuous, then the set of equicontinuous points is consistent with the set of all transitive points, so that it
is uniformly rigid \cite{GW}, the topological entropy is zero \cite{GM}. We known that if $(X,T)$ is minimal, then $(X,T)$ is either equicontinuous or sensitive
\cite{AY}; and if $(X,T)$ is transitive, then $(X,T)$ is either almost equicontinuous or sensitive \cite{AAB}. For the case of almost mean equicontinuous systems, the
authors in \cite{LTY1} showed that the set of transitive points is contained in the set of all mean equicontinuous points and there are examples in which they do not
coincide.

In \cite{LTY1}, the authors showed that if a dynamical system $(X,T)$  is minimal, then $(X,T)$  is either mean equicontinuous or mean sensitive, and if $(X,T)$  is
transitive, then $(X,T)$ is either almost mean equicontinuous or mean sensitive.  For amenable group actions, the authors obtained a dichotomy result related to almost
Weyl-mean equicontinuity and Weyl-mean sensitivity for when a group action is transitive and the action is either Weyl-mean sensitive or Weyl-mean equicontinuous if the
group action is a minimal system\cite{ZHL}.

In the setting of continuous bundle random dynamical system, instead of considering the iteration of a mapping, we study the successive continuous application of different
transformations chosen at random.

The paper is organized as follows: we begin in Section 2 by recalling some basic notations, definitions and results regarding continuous bundle random dynamical systems
or RDS simply. In Section 3 we introduce the concept and basic propositions of Weyl- mean equicontinuity for RDS. In Section 4 we introduce the concept and basic
propositions of Weyl- mean sensitivity for RDS and we obtain a dichotomy result related to Weyl-mean equicontinuity and Weyl-mean sensitivity for RDS when the corresponding
skew product transformation is minimal in the last part of the article.

\section{Preliminaries}

First, we review the basic knowledge of group actions and the definition of a random dynamical systems to be used in this paper. For more information on random dynamical
systems, refer to \cite[\textit{ p.\,27}]{DZ} and some well-known facts about ergodic theory, which can be found in many excellent books and papers
on the subject, e.g. \cite{A,DZ,KL,V1,V2}.

Let $\Omega\neq\emptyset$ be an abstract set, $\mathcal{F}$ a $\sigma-$algebra of subsets of $\Omega$ and $\mathbb{P}$ a probability measure on $\mathcal{F}$. The pair
$(\Omega,\mathcal{F})$ is called a measurable space and the triple $(\Omega,\mathcal{F},\mathbb{P})$ a probability space. A probability space is said to be complete if
the $\sigma-$algebra $\mathcal{F}$ contains all subsets of sets of probability 0.

By a measurable dynamical $G$-system (MDS) $(Y, \mathcal{D}, \nu, G)$ we mean a probability space $(Y, \mathcal{D}, \nu)$ and a group $G$ of invertible measure-preserving
transformations of $(Y,\mathcal{D}, \nu)$ with $e_{G}$ acting as the identity transformation.

\begin{definition}
The action $G\curvearrowright X$ is minimal if $X$ has no nonempty proper $G$ invariant closed subsets. A closed $G$-invariant set $A \subseteq X$ is minimal if the
restriction of the action to $A$ is minimal. The action $G\curvearrowright X$ is (topologically) transitive if for all nonempty open sets $U, V \subseteq  X$ there
exists an $s \in G$ such that $sU\cap V\neq\emptyset$.
\end{definition}

Let $(\Omega,\mathcal{F},\mathbb{P},G)$ denote an MDS, where $(\Omega,\mathcal{F},\mathbb{P})$ is a Lebesgue space. In particular, $(\Omega,\mathcal{F},\mathbb{P})$ is
complete and countably separated. Now let $(X,\mathcal{B})$ be a measurable space and $\mathcal{E}\in\mathcal{F}\times\mathcal{B}$. Then
$(\mathcal{E},(\mathcal{F}\times\mathcal{B})_{\mathcal{E}})$ forms naturally a measurable space with $(\mathcal{F}\times\mathcal{B})_{\mathcal{E}}$ the $\sigma-$algebra
of $\mathcal{E}$ given by restricting $\mathcal{F}\times\mathcal{B}$ over $\mathcal{E}$. Set $\mathcal{E}_{\omega}=\{x\in X:(\omega,x)\in \mathcal{E}\}$ for each
$\omega\in\Omega$. A bundle random dynamical system or random dynamical system (RDS) associated to $(\Omega,\mathcal{F},\mathbb{P},G)$ is a family
$$\mathbf{F}=\{F_{g,\omega}:\mathcal{E}_{\omega}\rightarrow \mathcal{E}_{g\omega}\mid g\in G, \omega\in \Omega\}$$ satisfying:

\begin{enumerate}
\item for each $\omega\in\Omega$, the transformation $F_{e_{G},\omega}$ is the identity over $\mathcal{E}_{\omega}$.

\item for each $g\in G$, the map $(\mathcal{E},(\mathcal{F}\times\mathcal{B})_{\mathcal{E}})\rightarrow (X,\mathcal{B})$, given by $(\omega,x)\mapsto F_{g,\omega}(x)$, is
      measurable and. \item for each $\omega\in\Omega$ and all $g_{1},g_{2}\in G,F_{g_{2},g_{1}\omega}\circ F_{g_{1},\omega}=F_{g_{2}g_{1},\omega}$(and so
      $F_{g^{-1},\omega}=(F_{g,g^{-1}\omega})^{-1}$ for each $g\in G$).
\end{enumerate}

In this case, $G$ has a natural measurable action on $\mathcal{E}$ with $(\omega,x)\rightarrow(g\omega,F_{g,\omega}x)$ for each $g\in G$, called the corresponding skew
product transformation.

Let the family $\mathbf{F}=\{F_{g,\omega}:\mathcal{E}_{\omega}\rightarrow \mathcal{E}_{g\omega}\mid g\in G, \omega\in \Omega\}$ be an RDS over
$(\Omega,\mathcal{F},\mathbb{P},G)$, where $X$ is a compact metric space with metric $d$ and equipped with the Borel $\sigma-$algebra $\mathcal{B}_{X}$. If for
$\mathbb{P}-$a.e. $\omega\in\Omega$, $\emptyset\neq\mathcal{E}_{\omega}\subseteq X$ is a compact subset and $F_{g,\omega}$ is a continuous map for each $g\in G$(and so
$F_{g,\omega}:\mathcal{E}_{\omega}\rightarrow \mathcal{E}_{g\omega}$ is a homeomorphism for $\mathbb{P}-$a.e. $\omega\in\Omega$ and each $g\in G$), then it is called a
continuous bundle RDS.

These concepts generalize the classical concepts of dynamical systems as follows. $(\Omega,\mathcal{F},\mathbb{P},G)$ is a trivial MDS, if $\Omega$ is a
singleton. For a compact metric space $X$, a continuous bundle RDS associated to a trivial MDS means that there is a topological $G-$ action $(K,G)$ for some non-empty
compact subset $K\subseteq X$, that is, the group $G$ acts on $K$ in the sense that there exists a family of homeomorphisms $\{F_{g}:g\in G\}$ of $K$ such that
$F_{g_{2}}\circ F_{g_{1}}=F_{g_{2}g_{1}}$ for all $g_{1},g_{2}\in G$ and $F_{e_{G}}$ acts as the identity transformation over $K$. The pair $(K,G)$ is also called a
topological dynamical $G-$systems (TDS). For more information on the theory of random systems see \cite{C,DGS,G,K,LQ,R,SH}.

In addition, let's recall the concept of Banach density. Let $E\subseteq G$, $\{F_{n}\}$ a F{\o}lner sequence of $G$. The upper density $\bar{d}(E)$ of $E$ respect to
$\{F_{n}\}$ by $$\bar{d}(E)=\limsup_{n\rightarrow\infty}\frac{|E\cap F_{n}|}{|F_{n}|}.$$
Similar, $\underline{d}(E)$, the lower density of $E$ respect to $\{F_{n}\}$, is $$\underline{d}(E)=\liminf_{n\rightarrow\infty}\frac{|E\cap F_{n}|}{|F_{n}|}$$
One say $E$ has density $d(E)$ if $\bar{d}(E)=\underline{d}(E)$, in which $d(E)$ is equal to this common value. The upper Banach density $BD^{*}(E)$ is
$$BD^{*}(E)=\limsup_{F}\frac{|F\cap E|}{|F|}$$
Similarly, lower Banach density $BD_{*}(E)$ and Banach density $BD(E)$.

\section{Besicovitch-,Weyl-,and Banach- mean equicontinuity for RDS}
In this section, we study the localization of mean equicontinuity. Inspired by the papers \cite{FGL}, \cite{LTY1} and \cite{ZHL}, we consider the concept of Weyl-mean
equicontinuity for random dynamical systems as follows. Before starting this section, for the convenience of using the content later in the article, we will make the
following marks. Let $t\in \Omega$, $x,y\in X$,
$$
\widetilde{d}(tx,ty)=\sup_{\omega\in\Omega}d(F_{t,\omega}x,F_{t,\omega}y).
$$
when $x,y\in \mathcal{E}_{\omega}$ for some $\omega\in G$; otherwise $\widetilde{d}_{t}(x,y)=\infty$.

\begin{definition}
Let $G$ be an amenable group and $\mathcal{F}=\{F_{n}\}_{n\in \mathbb{N}}$ be a F{\o}lner sequence of $G$. We call the continuous random dynamical system
$\mathbf{F}=\{F_{g,\omega}:\mathcal{E}_{\omega}\rightarrow\mathcal{E}_{g\omega}\mid g\in G, \omega\in \Omega\}$ is  {\it{Besicovitch-$\mathcal{F}$-mean equicontinuous}}
if for every $\varepsilon>0$, there exist $\delta>0$ such that
$$
D^{(r)}_{\mathcal{F}}(x,y):=\limsup_{n\rightarrow\infty}\frac{1}{|F_{n}|}\sum_{t\in F_{n}}\widetilde{d}(tx,ty)<\varepsilon,
$$
for all $x,y\in\mathcal{E}_{\omega}$ for some $\omega\in \Omega$ with $d(x,y)<\delta$.

The dependence on the F{\o}lner sequence immediately motivates the next definition. We say the continuous random dynamical system
$\mathbf{F}=\{F_{g,\omega}:\mathcal{E}_{\omega}\rightarrow\mathcal{E}_{g\omega}\mid g\in G, \omega\in \Omega\}$ is {\it {Weyl-mean equicontinuous}} (or simply WME) if for every
$\varepsilon>0$, there exists $\delta>0$ such that whenever $x,y\in \mathcal{E}_{\omega}$ for some $\omega\in \Omega$ with $d(x,y)<\delta$, we have
\begin{equation}\label{weya}
  D^{(r)}(x, y):=\sup \ \{D^{(r)}_{\mathcal{F}}(x, y) : \mathcal{F} \text {  is a F{\o}lner sequence}\}<\varepsilon.
\end{equation}
\end{definition}

\begin{definition}\cite{ZHL}
Let $G$ be an amenable group and $\mathcal{F}=\{F_{n}\}_{n\in\mathbb{N}}$ be a F{\o}lner sequence of $G$. We call the action $G\curvearrowright X$ is
{\it {Besicovitch-$\mathcal{F}$-mean equicontinuous}} if for every $\varepsilon>0$, there exists $\delta>0$ such that
$$
D_{\mathcal{F}}(x,y):=\limsup_{n\rightarrow\infty}\frac{1}{|F_{n}|}\sum_{g\in F_{n}}d(gx,gy)<\varepsilon,
$$
for all $x, y\in X$ with $d(x, y)<\delta$. The action $G\curvearrowright X$ is
{\it {Weyl-mean equicontinuous}} if for every $\varepsilon>0$, there exists $\delta>0$ such that for all $x, y\in X$ with $d(x, y)<\delta$ we have
\begin{equation}
  D(x, y):=\sup \ \{D_{\mathcal{F}}(x, y) : \mathcal{F} \  \text {  is a F{\o}lner sequence}\}<\varepsilon.
\end{equation}
\end{definition}

\begin{fact}
Let $X$ be a compact metric space, $G$ be a group and $(\Omega,\mathcal{F},\mathbb{P},G)$ is a trivial MDS. If the continuous random dynamical system
$\mathbf{F}=\{F_{g,\omega}:\mathcal{E}_{\omega}\rightarrow\mathcal{E}_{g\omega}\mid g\in G, \omega\in \Omega\}$ is {\it {Weyl-mean equicontinuous}}, then the action
$G\curvearrowright K$ is {\it {Weyl-mean equicontinuous}} for some non-empty compact subset $K\subseteq X$.
\end{fact}

\begin{definition}
Let $G$ be a countable discrete group, Fin$(G)$ be the family of all non-empty finite subsets of $G$ and $X$ be a compact metric space with metric $d$. If $x,y\in\mathcal{E}_{\omega}$ for some $\omega\in \Omega$, we denote
$$
\overline{D}^{(r)}(x,y)=\inf_{K\in {\rm {Fin}}(G)}\sup_{g\in G}\frac{1}{|K|}\sum_{t\in Kg}\widetilde{d}(tx,ty).
$$
So we call the continuous random dynamical system
$$\mathbf{F}=\{F_{g,\omega}:\mathcal{E}_{\omega}\rightarrow\mathcal{E}_{g\omega}\mid g\in G, \omega\in \Omega\}$$
is {\it{Banach mean equicontinuous}} or simply {\it {BME}} if for any $\varepsilon>0$, there is a positive real number $\delta>0$ such that
$\overline{D}^{(r)}(x,y)<\varepsilon$ whenever $x,y\in \mathcal{E}_{\omega}$ for some $\omega\in \Omega$ with $d(x,y)<\delta$.
\end{definition}

\begin{definition}\cite{ZHL}
Let $G$ be a discrete group and Fin$(G)$ be the family of all non-empty finite subsets of $G$. Let $X$ be a compact metric space with metric $d$. For $x,y\in X$, we denote
$$
\overline{D}(x,y)=\inf_{F\in {\rm {Fin}}(G)}\sup_{g\in G}\frac{1}{|F|}\sum_{t\in Fg}d(tx,ty).
$$
So we call the action $G\curvearrowright X$ is {\it {Banach-mean equicontinuous}} or simply {\it {B-mean equicontinuous}} if for any $\varepsilon>0$, there exists $\delta>0$
such that $\overline{D}(x,y)<\varepsilon$ whenever $d(x,y)<\delta$ for $x, y\in X$.
\end{definition}

\begin{fact}\label{BME}
Let $X$ be a compact metric space, $G$ be a group and $(\Omega,\mathcal{F},\mathbb{P},G)$ is a trivial MDS. If the continuous random dynamical system
$$\mathbf{F}=\{F_{g,\omega}:\mathcal{E}_{\omega}\rightarrow\mathcal{E}_{g\omega}\mid g\in G, \omega\in \Omega\}$$ is {\it{Banach mean equicontinuous}}, then and only the
action $G\curvearrowright K$ is {\it {Banach-mean equicontinuous}} for some non-empty compact subset $K\subseteq X$.
\end{fact}

\begin{fact}
When $G$ is a countable amenable group, the above two definitions are equivalent.
\end{fact}

In the following, for a random dynamical system associated to an infinite countable discrete amenable group action, we will see that the Banach pseudometric
$\overline{D}^{(r)}(\cdot,\cdot)$ equals to the Wely pseudometric $D^{(r)}(\cdot, \cdot)$.

\begin{theorem}\label{DD}
Let $G$ be a countable amenable group, $(\Omega,\mathcal{F},\mathbb{P},G)$ a  MDS and
$$
\mathbf{F}=\{F_{g,\omega}:\mathcal{E}_{\omega}\rightarrow\mathcal{E}_{g\omega}\mid g\in G, \omega\in \Omega\}
$$
is continuous random dynamical system. Then
\begin{equation*}
  D^{(r)}(x, y)=\overline{D}^{(r)}(x, y) \quad  {\text {for any pair  }}\ x,y\in X.
\end{equation*}
\end{theorem}

\begin{proof}

Let $x,y\in X$. If $x,y$ do not both belong to $\mathcal{E}_{\omega}$ for any $\omega\in \Omega$, $D^{(r)}(x, y)=\overline{D}^{(r)}(x, y)$ is obvious.

For the remaining part, we only need to prove that another situation holds.

Let $x, y\in \mathcal{E}_{\omega}$ for some $\omega\in \Omega$. Firstly, we show that $D^{(r)}(x, y)\leq\overline{D}^{(r)}(x, y)$.

Let $\varepsilon>0$. From the definition of $\overline{D}^{(r)}(x,y)$, there is a nonempty finite subset $K\in {\rm {Fin}}(G)$ such that
$$
\sup_{g\in G}\frac{1}{|K|}\sum_{t\in Kg}\widetilde{d}(tx,ty)<\overline{D}^{(r)}(x,y)+\varepsilon.
$$

Let  $\mathcal{F}=\{F_{n}\}_{n\in \mathbb{N}}$ be a F{\o}lner sequence of $G$. In the following we will show that
$$
\limsup_{n\rightarrow\infty}\frac{1}{|F_{n}|}\sum_{t\in F_{n}}\widetilde{d}(tx,ty)\leq\overline{D}^{(r)}(x, y)+\varepsilon.
$$

Given $g\in G$. For every $h\in F_{n}$, one has

\begin{equation*}
\begin{split}
\frac{1}{|K|}\sum_{s\in Khg}\widetilde{d}(sx,sy)&\leq\sup_{g'\in G}\frac{1}{|K|}\sum_{s\in Kg'}\widetilde{d}(sx,sy)\\
&<\overline{D}^{(r)}(x, y)+\varepsilon.
\end{split}
\end{equation*}
Thus it follows that
$$
\sum_{h\in F_{n}}\sum_{s\in Khg}\widetilde{d}(sx,sy)<|F_{n}||K|(\overline{D}^{(r)}(x,y)+\varepsilon).
$$

We denote by
$$\alpha(h,t)=\widetilde{d}((thg)x,(thg)y)$$ for $h\in F_{n}$ and $t\in K$. Then the above inequality can be re-written as
$$
\sum_{h\in F_{n}}\sum_{t\in K}\alpha(h,t)<|F_{n}||K|(\overline{D}^{(r)}(x, y)+\varepsilon).
$$
It is clear that there is $t'\in K$ such that
$$
\sum_{h\in F_{n}}\alpha(h,t')=\min\left\{\sum_{h\in F_{n}}\alpha(h,t): t\in K\right\}.
$$
Therefore, we get
$$
|K| \sum_{h\in F_{n}}\alpha(h,t')\leq\sum_{t\in K}\sum_{h\in F_{n}}\alpha(h,t)=\sum_{h\in F_{n}}\sum_{t\in K}\alpha(h,t)<|F_{n}||K|(\overline{D}^{(r)}(x, y)+\varepsilon).
$$
which implies
$$
\sum_{s\in t' F_{n} g} \widetilde{d}(sx,sy)=\sum_{h\in F_{n}}\alpha(h,t')<|F_{n}|(\overline{D}^{(r)}(x, y)+\varepsilon).
$$
Note that
\begin{equation*}
  \begin{split}
  \sum_{s\in F_{n} g}\widetilde{d}(sx,sy)& \leq \sum_{s\in t' F_{n} g}\widetilde{d}(sx,sy)+\sum_{s\in t' F_{n} g\triangle F_{n} g}\widetilde{d}(sx,sy) \\
    &<|F_n|(\overline{D}^{(r)}(x, y)+\varepsilon)+|t' F_{n} g\triangle F_{n} g|\cdot {\rm {diam}}(X)\\
    & =|F_n|(\overline{D}^{(r)}(x, y)+\varepsilon)+|t' F_{n}\triangle F_{n}|\cdot {\rm {diam}}(X)
\end{split}
\end{equation*}
where ${\rm {diam}}(X)$ is the diameter of the compact metric space $(X, d)$. Since $\{F_n\}$ is a F{\o}lner sequence, we have
\begin{equation*}
\begin{split}
D^{(r)}_{\mathcal{F}}(x, y) \leq \overline{D}^{(r)}(x, y)+\varepsilon+\limsup_{n\rightarrow\infty}\frac{|t' F_{n}\triangle F_{n}|}{|F_n|} {\rm {diam}}(X)=
\overline{D}^{(r)}(x, y)+\varepsilon.
\end{split}
\end{equation*}
From the arbitrariness of the F{\o}lner sequence $\{F_n\}$ we get
$$
D^{(r)}(x, y)\leq \overline{D}^{(r)}(x, y)+\varepsilon.
$$

The arbitrariness of $\varepsilon$ implies that
\begin{equation*}
  D^{(r)}(x, y)\leq \overline{D}^{(r)}(x, y).
\end{equation*}

Suppose that $D^{(r)}(x_0, y_0)<\overline{D}^{(r)}(x_0, y_0)$ for some $\omega\in \Omega$ and $x_0, y_0\in \mathcal{E}_{\omega}$. In what follows we will obtain a contradiction.

We choose two real number $\eta_1, \eta_2\in\mathbb{R}$ such that
\begin{equation}\label{MP006}
  D^{(r)}(x_0, y_0)<\eta_1<\eta_2<\overline{D}^{(r)}(x_0, y_0).
\end{equation}

Let $\{F_{n}\}_{n\in \mathbb{N}}$ be a F{\o}lner sequence of the amenable group $G$. Note that $F_{n}$ is a nonempty finite subset of $G$ for each $n\in\mathbb{N}$.
From the definition of $\overline{D}^{(r)}(x_{0}, y_{0})$, we have

$$
\sup_{g\in G}\frac{1}{|F_{n}|}\sum_{t\in F_{n}g}\widetilde{d}(tx_{0},ty_{0})\geq\overline{D}^{(r)}(x_{0}, y_{0})>\eta_2.
$$

Thus, for each $n\in\mathbb{N}$, there exists $g_{n}\in G$ such that

$$
\frac{1}{|F_{n}g_{n}|}\sum_{t\in F_{n}g_{n}}\widetilde{d}(tx_{0},ty_{0})>\eta_2.
$$

Set $H_{n}=F_{n}g_{n}$. Since $\{H_{n}\}_{n\in \mathbb{N}}$ is also a (left) F{\o}lner sequence of $G$, we get

$$
\eta_1>D^{(r)}(x_0, y_0)\geq D^{(r)}_{\{H_{n}\}_{n\in\mathbb{N}}}(x_0,y_0)=\limsup_{n\rightarrow\infty}\frac{1}{|H_{n}|}\sum_{g\in H_{n}}\widetilde{d}(gx_{0},gy_{0})
\geq\eta_2.
$$

This is a contradiction. Hence the theorem is proved.
\end{proof}

\begin{remark}
The definition of Banach mean equicontinuity may seem to depend on the particular choice of the metric. However, by compactness of the underlying space $X$ is turns out
to hold for one metric if and only if it holds for any metric (which induces the same topology).
\end{remark}

From the above Theorem \ref{DD}, it follows that the concepts of Banach-and Weyl-mean equicontinuity are equivalent for RDS.

\begin{Corollary}\label{Cor1}
Let $G$ be a countable amenable group, let $X$ be a compact metric space, then RDS $\mathbf{F}$ is Banach-mean equicontinuous if and only if the RDS is Weyl-mean
equicontinuous.
\end{Corollary}

When studying dynamical systems with discrete spectrum, Fomin \cite{F} introduced a notion called stable in the mean in the sense of Lyapunov or simply mean-L-stable.
Fomin proved that if a minimal system is mean-L-stable then it is uniquely ergodic. Li, Tu and Ye \cite{LTY1} showed that every ergodic invariant measure on a mean-L-stable
system has discrete spectrum. Next, we will introduce mean-L-stable of a RDS for an amenable group action.

\begin{definition}
A continuous random dynamical system $\mathbf{F}=\{F_{g,\omega}:\mathcal{E}_{\omega}\rightarrow\mathcal{E}_{g\omega}\mid g\in G, \omega\in \Omega\}$ is called
{\it {stable in the mean in the sense of Lyapunov}} or simply {\it {mean-L-stable}} if for every $\varepsilon>0$ there is a positive real number $\delta>0$ such that
whenever $x,y\in \mathcal{E}_{\omega}$ for some $\omega\in \Omega$ with $d(x,y)<\delta$ implies
$$\widetilde{d}(gx,gy)<\varepsilon$$
for all $g\in G$ except a set of upper Banach density less than $\varepsilon$.
\end{definition}

\begin{lemma}\label{lem2}
Suppose that $G$ is an amenable group. Then the continuous random dynamical system
$\mathbf{F}=\{F_{g,\omega}:\mathcal{E}_{\omega}\rightarrow\mathcal{E}_{g\omega}\mid g\in G, \omega\in \Omega\}$
is {\it{BME}} if and only if it is mean-L-stable.
\end{lemma}

\begin{proof}
Assume that the continuous random dynamical system $\mathbf{F}=\{F_{g,\omega}:\mathcal{E}_{\omega}\rightarrow\mathcal{E}_{g\omega}\mid g\in G, \omega\in \Omega\}$ is
{\it{BME}}. By Corollary  \ref{Cor1}, Let $\varepsilon>0$, there is a positive real number $\delta>0$ such that whenever $x,y\in\mathcal{E}_{\omega}$ for some
$\omega\in \Omega$ with $d(x,y)<\delta$,
$$
\sup_{\{F_{n}\}}\limsup_{n\rightarrow\infty}\frac{1}{|F_{n}|}\sum_{g\in F_{n}}\widetilde{d}(gx,gy)<\varepsilon^{2}
$$
where the sumpremum is taken over all F{\o}lner sequence $\{F_{n}\}$ of $G$.

Given two points $x,y\in \mathcal{E}_{\omega}$ for some $\omega\in \Omega$ with $d(x,y)<\delta$. Let
$$E=\left\{g\in G:\widetilde{d}(gx,gy)\geq\varepsilon\right\}.$$
It follows that

\begin{equation*}
\begin{split}
\varepsilon \cdot {\rm {BD}}^{*}(E)&=\sup_{\{F_{n}\}}\limsup_{n\rightarrow\infty}\frac{\varepsilon|F_{n}\cap E|}{|F_{n}|}\\
&\leq\sup_{\{F_{n}\}}\limsup_{n\rightarrow\infty}\frac{1}{|F_{n}|}\sum_{g\in F_{n}}\widetilde{d}(gx,gy)\\
&<\varepsilon^{2}
\end{split}
\end{equation*}
and then BD$^{*}(E)<\varepsilon$, which implies the continuous random dynamical system $\mathbf{F}$ is mean-L-stable.

Conversely, assume that the continuous random dynamical system
$$\mathbf{F}=\{F_{g,\omega}:\mathcal{E}_{\omega}\rightarrow\mathcal{E}_{g\omega}\mid g\in G, \omega\in \Omega\}$$
is mean-L-stable. Fix a positive number $\varepsilon>0$ and choose a positive real number $\eta<\frac{\varepsilon}{2(1+{\rm
{diam}}(X))}$, there is a positive real number $\delta>0$ such that whenever $x,y\in X$ with $x,y\in \mathcal{E}_{\omega}$ for
some $\omega\in \Omega$ and $d(x,y)<\delta$ implies
$$\widetilde{d}(gx,gy)<\eta$$
for all $g\in G$ except a set of upper Banach density less than $\eta$.

Given two points $x,y\in \mathcal{E}_{\omega}$ for some $\omega\in\Omega$ and $d(x,y)<\delta$. Let $\{F_{n}\}_{n\in \mathbb{N}}$ be a F{\o}lner sequence of $G$ and
$$E=\left\{g\in G:\widetilde{d}(gx,gy)\geq\eta\right\}.$$
It follows that BD$^{*}(E)<\eta$ and

\begin{equation*}
\begin{split}
&D^{(r)}_{\{F_{n}\}_{n\in \mathbb{N}}}(x, y)\\
&\leq\limsup_{n\rightarrow\infty}\frac{1}{|F_{n}|}\left[\sum_{g\in F_{n}\setminus E}\widetilde{d}(gx,gy)+\sum_{g\in E\cap F_{n}}\widetilde{d}(gx,gy)\right]\\
  & \leq \limsup_{n\rightarrow\infty}\frac{1}{|F_{n}|}\left[\sum_{g\in F_{n}\setminus E}\widetilde{d}(gx,gy)+|E\cap F_{n}|\cdot{\rm {diam}(X)}\right]\\
  & \leq \limsup_{n\rightarrow\infty}\frac{1}{|F_{n}|}\left[\sum_{g\in F_{n}\setminus E}\eta+|E\cap F_{n}|\cdot{\rm {diam}(X)}\right]\\
  &\leq\limsup_{n\rightarrow\infty}\frac{1}{|F_{n}|}|F_{n}|\eta+\limsup_{n\rightarrow\infty}\frac{|E\cap F_{n}|}{|F_{n}|}\rm {diam}(X)\\
  &\leq \eta+{\rm {diam}(X)}\cdot {\rm {BD}}^{*}(E)<(1+{\rm {diam}(X)})\eta<\frac{\varepsilon}{2},
\end{split}
\end{equation*}
which implies the continuous random dynamical system $\mathbf{F}=\{F_{g,\omega}:\mathcal{E}_{\omega}\rightarrow\mathcal{E}_{g\omega}\mid g\in G, \omega\in \Omega\}$
is {\it{BME}}.
\end{proof}

According to Corollary \ref{Cor1} and Lemma \ref{lem2}, we can get the following conclution:
\begin{theorem}
Let $G$ be a countable amenable group. $\mathbf{F}$ be a RDS over $(\Omega,\mathcal{F},\mathbb{P},G)$. Then the following three statements are equivalent:
\begin{enumerate}
\item RDS $\mathbf{F}$ is WME;

\item RDS $\mathbf{F}$ is BME;

\item RDS $\mathbf{F}$ is mean-L-stable.
\end{enumerate}

\end{theorem}

In determination dynamical systems, we know that if every point in $X$ is an equicontinuous point then by the compactness of $X$ the dynamical system $(X, T )$ is
equicontinuous. $(X, T )$ is mean equicontinuous if and only if every point in $X$ is mean equicontinuous \cite{LTY1}. An amenable action $G\curvearrowright X$ is Banach-mean
equicontinuous if and only if every point in $X$ is a Banach-mean equicontinuous point \cite{ZHL}. At the end of this secion, we will introduce the Weyl-mean equicontinuous
point for RDS $\mathbf{F}$.

We call a point $x\in X$ is  {\it {Besicovitch-$\mathcal{F}$-mean equicontinuous}} if for every $\varepsilon>0$, there exists $\delta>0$ such that
$D^{(r)}_{\mathcal{F}}(x, y)<\varepsilon,$ for all $x,y\in \mathcal{E}_{\omega}$ for some $\omega\in \Omega$ with $d(x,y)<\delta$.

Let $x\in X$ is {\it {Weyl-mean equicontinuous}} if for every $\varepsilon>0$, there exists $\delta>0$ such that whenever
$y\in \mathcal{E}_{\omega}$ for some $\omega\in\Omega$, with $d(x,y)<\delta$, we have
\begin{equation}\label{wey2}
  D^{(r)}(x, y)<\varepsilon.
\end{equation}

Let $\omega\in\Omega$, set $B_{\omega}(x,\delta)=B(x,\delta)\cap\mathcal{E}_{\omega}$ where $B(x,\delta)$ is an open ball centered at $x$ of radius $\delta$. From the definition of RDS, it can be observed that there is
a set $A$ with $A\subseteq \Omega$ and $\mathbb{P}(A)=1$ such that $\emptyset\neq\mathcal{E}_{\omega}\subseteq X$ is a compact subset and $F_{g,\omega}$ is a continuous map
for each $g\in G$ and $\omega\in A$.

If $\Omega$ finite, by the compactness of $\mathcal{E}_{\omega}$ for $\omega\in\Omega$, $\mathbf{F}$ is Weyl-mean equicontinuous if and only if every point in
$\bigcup_{\omega\in A}\mathcal{E}_{\omega}$ is Weyl mean equicontinuous.  We say the continuous random dynamical system
$\mathbf{F}=\{F_{g,\omega}:\mathcal{E}_{\omega}\rightarrow\mathcal{E}_{g\omega}\mid g\in G, \omega\in \Omega\}$ is {\it{almost Weyl-mean equicontinuous}} if
the dynamic system  has at least one Weyl-mean equicontinuous point.

Let $\mathbb{E}$ denote the set of all Weyl-mean equicontinuous points. Let $\omega\in\Omega$, for every $\varepsilon>0$,
\begin{equation}\label{xxx}
 \mathbb{E}_{\omega,\varepsilon}=\bigg\{x\in \mathcal{E}_{\omega}: \exists \, \delta>0,\, \forall\, y, z\in B_{\omega}(x,\delta),  D^{(r)}(y, z)<\varepsilon \bigg\}
\end{equation}

For the Weyl-mean equicontinuous points we have the following proposition.
\begin{prop}\label{g}
Let $G$ be a countable amenable group, $\mathbf{F}=\{F_{g,\omega}:\mathcal{E}_{\omega}\rightarrow\mathcal{E}_{g\omega}\mid g\in G, \omega\in \Omega\}$ be a continuous random
dynamical system and $\varepsilon>0$. Then $\mathbb{E}_{\omega,\varepsilon}$ is open. Moreover, if $\Omega$ is finite set, $\mathbb{E}=\bigcap_{m=1}^{\infty}\bigcup_{\omega\in A}\mathbb{E}_{\omega,\frac{1}{m}}$ is a $G_{\delta}$
subset of $X$.
\end{prop}

\begin{proof}
Let $\omega\in\Omega$, $\varepsilon>0$ and $x\in\mathbb{E}_{\omega,\varepsilon}$. Choose $\delta>0$ satisfying the condition from the definition of
$\mathbb{E}_{\omega,\varepsilon}$ for $x$. Fix $y\in B_{\omega}(x, \delta/2)$. If $z,w\in B_{\omega}(y, \delta/2)$, then $z,w\in B_{\omega}(x,\delta)$, so
$D^{(r)}(z, w)<\varepsilon$. This shows that $B_{\omega}(x, \delta/2)\subseteq \mathbb{E}_{\omega,\varepsilon}$ and hence $\mathbb{E}_{\omega,\varepsilon}$ is open.

If $x\in\bigcap_{m=1}^{\infty}\bigcup\limits_{\omega\in A}\mathbb{E}_{\omega, \frac{1}{m}}$, then clearly $x\in \mathbb{E}_{\omega,\frac{1}{m}}$
for all $m$ and for some $\omega\in\Omega$, so $x\in E$.

Conversely, if $x\in \mathbb{E}$, $m\geq1$, there exists $\delta>0$ such that $D^{(r)}(x, y)<1/2m$ for all $y\in B(x,\delta)$. Let $\omega\in A$. If $y,z\in
B_{\omega}(x,\delta)$, then
\begin{equation*}
D^{(r)}(y, z)\leq D^{(r)}(y, x)+D^{(r)}(x, z)<\frac{1}{m}.
\end{equation*}
Thus $x\in \mathbb{E}_{\omega,\frac{1}{m}}$. Therefore we get $\mathbb{E}=\bigcap_{m=1}^{\infty}\bigcup\limits_{\omega\in A}\mathbb{E}_{\omega,\frac{1}{m}}$.
Hence the proof is completed.
\end{proof}

\begin{fact}
Let $X$ be a compact metric space, $G$ be a group and $(\Omega,\mathcal{F},\mathbb{P},G)$ is a trivial MDS. The above results are consistent with the proposition 4.6 in
the paper \cite{ZHL}.
\end{fact}

Let $X$ be a compact metric space. Recall that a subset set of $X$ is called {\it {residual}} if it contains the intersection of a countable collection of dense open sets.
By the Baire category theorem, a residual set is also dense in $X$. If $\Omega$ is a singleton, we can draw the following conclusion:

\begin{prop}
Let $G$ be a countable amenable group, $\Omega$ is a singleton. let $X$ be a compact metric space. Let $G\curvearrowright \mathcal{E}$ be a transitive action. Then there
exist a compact subset $K\subseteq X$ satisfies
\begin{enumerate}

\item $G\curvearrowright K$ be a transitive action;

\item The set of Weyl-mean equieontimlous points $\mathbb{E}$ is either empty or residual. If, in addition, the action $G\curvearrowright K$ is almost Weyl-mean
      equicontinuous, then every transitive point is Weyl-mean equicontinuous;

\item If the action $G\curvearrowright \mathcal{E}$ is minimal i.e. the action $G\curvearrowright K$ is minimal and almost Weyl-mean equicontinuous, then
      $G\curvearrowright K$ is Weyl-mean equicontinuous.
\end{enumerate}

\end{prop}
\section{Weyl-mean sensitivity for RDS}

In this section, we will introduction Weyl-mean sensitive point and Weyl-mean sensitivity for RDS $\mathbf{F}$ associated to an infinite countable discrete amenable group action.
Let $\omega\in \Omega$, $x,y\in X$, $\mathcal{F}=\{F_{n}\}_{n\in \mathbb{N}}$ be a F{\o}lner sequence of $G$, we  define
$$
\hat{D}^{(r)}_{\mathcal{F},\omega}(x,y):=\limsup_{n\rightarrow\infty}\frac{1}{|F_{n}|}\sum_{t\in F_{n}}d(F_{t,\omega}x,F_{t,\omega}y),
$$
when $x,y\in\mathcal{E}_{\omega}$ for some $\omega\in\Omega$; otherwise $\hat{D}^{(r)}_{\mathcal{F},\omega}(x,y)=\infty$. Set
\begin{equation}\label{6}
  \hat{D}^{(r)}_{\omega}(x,y)=\sup_{\mathcal{F}}\hat{D}^{(r)}_{\mathcal{F},\omega}(x,y),
\end{equation}
and
$$
\hat{D}^{(r)}(x,y)=\sup_{\omega\in \Omega}\hat{D}^{(r)}_{\omega}(x,y).
$$

Note that
$$
\hat{D}^{(r)}_{\mathcal{F},\omega}(x,y)\leq D^{(r)}_{\mathcal{F}}(x,y)
$$
which implies that for $\omega\in \Omega$ and $x,y\in X$,

\begin{equation}
  \hat{D}^{(r)}_{\omega}(x,y)\leq D^{(r)}(x,y).
\end{equation}

Thus, for all $x,y\in X$,
$$
\hat{D}^{(r)}(x,y)\leq D^{(r)}(x,y),
$$
when $\Omega=\{\omega\}$, $\hat{D}^{(r)}(x,y)= D^{(r)}(x,y)$.

\begin{definition}
Let $\mathbf{F}$ a continuous random dynamical system. We call the point $x\in X$ is  {\it {Weyl-mean sensitive point}} for $\mathbf{F}$, there exist
$\omega\in \Omega$ and $\delta>0$, such that for every $\varepsilon>0$, there is $y\in\mathcal{E}_{\omega}\cap B(x,\varepsilon)$ satisfying
$$
 \hat{D}_{\omega}^{(r)}(x, y)>\delta.
$$
It is clear that a point is either Weyl mean equicontinuous or Weyl mean sensitive.

Here the definition of the function $\dot{D}^{(r)}(\cdot\,,\cdot)$ please refer to (\ref{weya}).

We call a continuous random dynamical system $\mathbf{F}$ is {\it {Weyl-mean sensitive}} if there exist $\delta>0$ such that every
$\varepsilon>0$, there are $\omega\in \Omega$ and $y\in \mathcal{E}_{\omega}\cap B(x,\varepsilon)$ satisfying
$$
 \hat{D}_{\omega}^{(r)}(x, y)>\delta.
$$
\end{definition}

\begin{definition}\cite{ZHL}
Let $G\curvearrowright X$ be a continuous action and $x\in X$ be a point. We call the point $x$ is  {\it {Weyl-mean sensitive point}} if there exists $\delta>0$
such that for every $\varepsilon>0$, there is $y\in B(x,\varepsilon)$ satisfying
$$
D(x, y)>\delta.
$$
\end{definition}

Regarding the relation between the Weyl-mean sensitive point of random dynamical systems and Weyl-mean sensitive point of group action systems, we can draw the
following conclusions.

\begin{fact}
Let $X$ be a compact metric space, $G$ be a group and $(\Omega,\mathcal{F},\mathbb{P},G)$ is a trivial MDS. If $x\in X$ is {\it {Weyl-mean sensitive point}} for
the continuous random dynamical system
$$\mathbf{F}=\{F_{g,\omega}:\mathcal{E}_{\omega}\rightarrow\mathcal{E}_{g\omega}\mid g\in G, \omega\in \Omega\},$$ then $x\in K$
is {\it {Weyl-mean sensitive point}} for an action $G\curvearrowright K$ for some non-empty compact subset $K\subseteq X$.
\end{fact}

\begin{prop}\label{Prop6}
Let $G$ be a countable amenable group and $X$ a compact metric space.
$$\mathbf{F}=\{F_{g,\omega}:\mathcal{E}_{\omega}\rightarrow\mathcal{E}_{g\omega}\mid g\in G, \omega\in \Omega\}$$ be continuous random dynamical system.
If there exist $\omega\in A$ and $\delta>0$ such that for every non-empty open subset $U_{\omega}$ of
$\mathcal{E}_{\omega}$, there are $g\in G, U'\subseteq \mathcal{E}_{g\omega}$ (there is $U_{0}\subseteq X$ satisfies
$U_{\omega}=U_{0}\cap\mathcal{E}_{\omega},U'=U_{0}\cap\mathcal{E}_{g\omega}$), there are $x_{0},y_{0}\in U'$ satisfying $\hat{D}_{g\omega}^{(r)}(x_{0}, y_{0})>\delta$.
Then the continuous random dynamical system $\mathbf{F}$ is Weyl-mean sensitivity.
\end{prop}

\begin{proof}
Suppose that there exist $\omega\in A$ and $\delta>0$ such that for every non-empty open subset $U_{\omega}$ of
$\mathcal{E}_{\omega}$, there are $g\in G, U'\subseteq \mathcal{E}_{g\omega}$ (there is $U_{0}\subseteq X$ satisfies
$U_{\omega}=U_{0}\cap\mathcal{E}_{\omega},U'=U_{0}\cap\mathcal{E}_{g\omega}$), there are $x_{0},y_{0}\in U'$ satisfying $\hat{D}_{g\omega}^{(r)}(x_{0}, y_{0})>2\delta$.

Let $(\omega,x)\in \mathcal{E}$ and $\varepsilon>0$, then $B(x,\varepsilon)\neq\emptyset$ and $B(x,\varepsilon)\cap\mathcal{E}_{\omega}$ is open subset of
$\mathcal{E}_{\omega}$. Then there exist $g\in G$ and $y,z\in B(x,\varepsilon)\cap\mathcal{E}_{g\omega}$ satisfying $\hat{D}_{g\omega}^{(r)}(y,z)>2\delta. $

Thus we have either $\hat{D}_{g\omega}^{(r)}(x, y)>\delta $ or $\hat{D}_{g\omega}^{(r)}(x,z)>\delta. $ which implies that the continuous random dynamical system
$\mathbf{F}$ is Weyl-mean sensitivity.
\end{proof}

\begin{prop}\label{Prop7}
Let $G$ be a countable amenable group and $X$ a compact metric space. Let the skew product transformation be transitive. If $x_{0}$ is a Weyl-mean sensitivity point and
$\delta_{0}$ satisfying the condition from the definition of Weyl-mean sensitivity point $x_{0}$ on $\mathcal{E}_{\omega_{0}}$ for transitive point
$(\omega_{0},x_{0})\in \mathcal{E}$, then the continuous random dynamical system $\mathbf{F}$ is a Weyl-mean sensitivity.
\end{prop}

\begin{proof}

By the assumption that $x_{0}$ a Weyl-mean sensitivity point and $\delta_{0}$ satisfying the condition from the
definition of Weyl-mean sensitivity point $x_{0}$ on $\mathcal{E}_{\omega_{0}}$. Given a nonempty open subset $U$ of $\mathcal{E}_{\omega_{0}}$, there exist
$U'\in\mathcal{B}$ such that
$U=U'\cap\mathcal{E}_{\omega_{0}}$. Then $(\Omega\times U')\cap\mathcal{E}$ is an open set of  $\mathcal{E}$. there is $s\in G$
such that $s(\omega_{0},x_{0})\in(\Omega\times U')\cap\mathcal{E}$ this since the skew product transformation be transitive and $(\omega_{0},x_{0})\in\mathcal{E}$
is a transitive point. And so, $F_{s,\omega_{0}}x_{0}\in U'\cap\mathcal{E}_{s\omega_{0}}$. Furthermore, as $x_{0}\in F_{s,\omega_{0}}^{-1} U'$ and
$F_{s,\omega_{0}}^{-1} U'$ is open, there is $\epsilon>0$ such that $B(x_{0},\epsilon)\subseteq F_{s,\omega_{0}}^{-1}U'$, that is,
$F_{s,\omega_{0}}B(x_{0},\epsilon)\subseteq U'$. By the assumption that $x_{0}$ a Weyl-mean sensitivity point and $\delta_{0}$ satisfying the condition from the
definition of Weyl-mean sensitivity point $x_{0}$, then there exists $y_{0}\in
\mathcal{E}_{\omega_{0}}\cap B(x_{0},\epsilon)$ satisfying
$$
 \hat{D}^{(r)}_{\omega_{0}}(x_{0}, y_{0})>\delta.
$$

By the definition of $\hat{D}^{(r)}_{\omega_{0}}(x_{0}, y_{0})$, there is a (left) F{\o}lner sequence $\mathcal{F}=\{F_n\}_{n\in\mathbb{N}}$ of $G$ such that
\begin{equation*}
 \hat{D}^{(r)}_{\mathcal{F},\omega_{0}}(x_{0}, y_{0})>\delta.
\end{equation*}

Let $u=F_{s,\omega_{0}}x_{0}$, $v=F_{s,\omega_{0}}y_{0}$. Noting that $\mathcal{F}s=\{F_n s^{-1}\}_{n\in\mathbb{N}}$ being also a (left)
F{\o}lner sequence and $u, v\in U'\subseteq\mathcal{E}_{s\omega_{0}}$, then
\begin{equation*}
\hat{D}^{(r)}_{\omega_{0}}(u, v)\geq \hat{D}^{(r)}_{\mathcal{F}s^{-1},s\omega_{0}}(u,
v)=\hat{D}^{(r)}_{\mathcal{F}s^{-1},s\omega_{0}}(F_{s,\omega_{0}}x_{0},
F_{s,\omega_{0}}y_{0}) =\hat{D}^{(r)}_{\mathcal{F},\omega_{0}}(x_{0}, y_{0})>\delta.
\end{equation*}
Therefore the continuous random dynamical system $\mathbf{F}$ is
Weyl-mean sensitivity by Proposition \ref{Prop6}.
\end{proof}
\section{Main result}

We known that if a dynamical system $(X, T )$ is minimal then $(X, T )$ is either mean equicontinuous or mean sensitive, and if $(X, T )$ is transitive then $(X, T )$ is
either almost mean equicontinuous or mean sensitive \cite{LTY1}. If the action $G\curvearrowright X$ is transitive, then the action $G\curvearrowright X$ is either almost
Weyl-mean equicontinuous or Weyl-mean sensitive. Let $G\curvearrowright X$ be a minimal system, then the action $G\curvearrowright X$ is either Weyl-mean sensitive or
Weyl-mean equicontinuous \cite{ZHL}.

In this section, it turns out that if the action $G\curvearrowright \mathcal{E}$  is minimal when $\Omega$ is finite, then RDS is either Weyl-mean equicontinuous or Weyl-mean sensitive.

\begin{Corollary}
Let $G$ be a countable amenable group and $X$ a compact metric space. If the skew product transformation $G\curvearrowright\mathcal{E}$ is minimal, then the
continuous random dynamical system $\mathbf{F}$ is either Weyl-mean equicontinuous or Weyl-mean sensitivity when $\Omega$ is finite.
\end{Corollary}

\begin{proof}
If $x_{0}$ is a Weyl-mean sensitivity point and $\delta_{0}$ satisfying the condition from the definition of Weyl-mean sensitivity point $x_{0}$ for transitive point
$(\omega_{0},x_{0})\in \mathcal{E}$, then the continuous random dynamical system $\mathbf{F}$ is Weyl-mean sensitivity by Proposition \ref{Prop7}. If $x_{0}$ is not
a Weyl-mean sensitivity point, then it is a Weyl-mean equicontinuous point. So the continuous random dynamical system $\mathbf{F}$ is Weyl-mean equicontinuous.
\end{proof}

\begin{fact}\label{tran}
Let $X$ be a compact metric space, $G$ be a group and $(\Omega,\mathcal{F},\mathbb{P},G)$ is a trivial MDS. If the natural measurable action $G$ on $\mathcal{E}$ is
minimal, then the action $G\curvearrowright K$ is topological minimal for some non-empty compact subset $K\subseteq X$.
\end{fact}

We end this article with a question.

\begin{question}
Let $G$ be a countable discrete amenable group, $(\Omega,\mathcal{F},\mathbb{P},G)$ an MDS,
$$\mathbf{F}=\{F_{g,\omega}:\mathcal{E}_{\omega}\rightarrow \mathcal{E}_{g\omega}\mid g\in G, \omega\in \Omega\}$$
be an RDS over $(\Omega,\mathcal{F},\mathbb{P},G)$, $\{F_{n}\}_{n\in \mathbb{N}}$ be a F{\o}lner sequence of $G$.
Then for $x,y\in \mathcal{E}_{\omega}$ for some $\omega\in A$, does the following formula make sense?
$$\limsup_{n\rightarrow\infty}\frac{1}{|F_{n}|}\sum_{t\in F_{n}}\int_{\Omega}d(F_{t,\omega}x,F_{t,\omega}y)dP(\omega).$$
\end{question}

Observe that if the above formula makes sense, we denote it by $\tilde{D}_{\mathcal{F}}^{(r)}(x,y)$, it is not hard to find
$\tilde{D}_{\mathcal{F}}^{(r)}(x,y)=D_{\mathcal{F}}^{(r)}(x,y)$.



\bibliographystyle{amsplain}

\end{document}